\documentclass{fic-l}
\usepackage{latexsym}
\usepackage{amssymb}
\usepackage{amsthm}
\usepackage{times}
\usepackage[all]{xy}
\usepackage{color}

\title{A survey of Galois theory of curves in characteristic $p$}
\author{Rachel Pries and Katherine Stevenson}
\date{\today}

\newtheorem{theorem}{Theorem}[section]
\newtheorem{lemma}[theorem]{Lemma}
\newtheorem{proposition}[theorem]{Proposition}
\newtheorem{corollary}[theorem]{Corollary}

\newtheorem{predefinition}[theorem]{Definition}

\newtheorem{preremark}[theorem]{Remark}

\newtheorem{prenotation}[theorem]{Notation}

\newtheorem{preexample}[theorem]{Example}
\newenvironment{example}{\begin{preexample}\rm}{\end{preexample}}
\newtheorem{preclaim}[theorem]{Claim}

\newtheorem{prequestion}[theorem]{Question}
\newenvironment{question}{\begin{prequestion}\rm}{\end{prequestion}}
\makeatletter 
\def\emppsubsection{\@startsection{subsection}{2}{\z@}{-3.25ex plus -1ex minus -.2ex}{-1em}{\bf}}

\makeatother

\newcommand \X {{\mathcal X}}
\newcommand \Y {{\mathcal Y}}
\newcommand  \B {{\mathcal B}}
\newcommand \CL {{\mathcal L}}
\newcommand \CO {{\mathcal O}}

\newcommand \PP {{\mathbb P}^1}
\newcommand \ZZ {{\mathbb Z}}
\newcommand  \FF {{\mathbb F}}
\newcommand \CC {{\mathbb C}}
\newcommand \GG {{\mathbb G}}
\newcommand  \QQ {{\mathbb Q}}
\newcommand  \NN {{\mathbb N}}
\newcommand  \Aa {{\mathbb A}^1}

\newcommand \Aut {\mathop{\rm Aut}}
\newcommand \Spec {\mathop{\rm Spec}}
\newcommand \Gal {\mathop{\rm Gal}}
\newcommand \dv  {\mathop{\rm div}}
\newcommand{\Ga}{({\mathbb Z}/\ell {\mathbb Z})^a \rtimes {\mathbb Z}/p {\mathbb Z}}

\begin{document}
\maketitle

\begin{abstract}
This survey is about Galois theory of curves in characteristic $p$, a topic which
has inspired major research in algebraic geometry and number theory and which contains many open questions.
We illustrate important phenomena which occur for covers of curves in characteristic $p$.
We explain key results on the structure of fundamental groups.
We end by describing areas of active research and giving two new results about the 
genus and $p$-rank of certain covers of the affine line. 
\end{abstract}


\thanks{This project was initiated at the workshop WIN Women in Numbers in November 2008.  The authors would like to thank the Banff International Research Station for hosting the workshop and the National Security Agency, the Fields Institute, the Pacific Institute for the Mathematical Sciences, Microsoft Research, and University of Calgary for their financial support.  The first author was partially supported by NSF grant 07-01303.}

\section{Introduction}
The purpose of this paper is to introduce the reader to the topic of Galois theory of curves in characteristic $p$.
Since topological methods are no longer applicable, this topic has inspired major research in algebraic geometry and number theory, to recapture some information on the structure of fundamental groups.  In spite of these advances,
there are still many fascinating open questions on this topic.   

In section \ref{Sfalse}, we recall some Galois theoretic facts for complex curves which are meaningless or false for curves in characteristic $p$.
Section \ref{Sexample} contains some crucial examples of Kummer and Artin-Schreier covers of curves in characteristic $p$.
The main algebraic definitions of objects such as the fundamental group, higher ramification groups, and Jacobians,
can be found in Section \ref{Sdefine}.
In Section \ref{Stheory}, we outline the proofs of several major results, including some of 
the contributions of Grothendieck, Harbater, Pop, Raynaud, Serre, and Tamagawa.

Finally, in Section \ref{Sopen}, we describe a few areas of active research 
involving embedding problems and arithmetic invariants of Galois covers defined over 
an algebraically closed field of characteristic $p>0$.
We prove two new results on these topics.
To describe the results, let $\ell$ be a prime distinct from $p$ and let $a$ be the order of $\ell$ modulo $p$.
Let $L$ be an $\ell$-group whose maximal elementary abelian quotient is $(\ZZ/\ell)^a$.
Let $G$ be a semi-direct product $L \rtimes \ZZ/p$.
In Proposition \ref{Pgenus}, we prove that the smallest genus which occurs for a 
(wildly ramified) $G$-Galois cover $\phi:W \to \PP_k$ branched only at $\infty$ is $g_W=1+|L|(p-3)/2$.
This result can be viewed as the solution to an embedding problem with prescribed ramification conditions.
In Proposition \ref{Pprank}, when $L \simeq (\ZZ/\ell)^a$, we prove that $W$ can be chosen such that 
its Jacobian $J_W$ has $p$-rank $s_W=(\ell^a-1)(p-3)/2$ and furthermore such that 
the $p$-torsion $J_W[p]$ decomposes completely into $s_W$ copies of $\ZZ/p \oplus \mu_p$ and 
$(p-1)/2$ copies of $E_{ss}[p]$, the $p$-torsion group scheme of a supersingular elliptic curve.
In particular, the Newton polygon of $J_W$ only has slopes $0$, $1/2$, and $1$.
The result is interesting because this combination of arithmetic invariants is somewhat unusual.

\section{Facts about Galois covers of complex curves} \label{Sfalse}
Here are some of the basic properties of Galois covers of complex curves that are false for covers of curves 
defined over a field of characteristic $p>0$. 
Suppose $\X$ is a smooth connected projective complex curve, i.e., a Riemann surface, of genus $g$.
Suppose $\B \subset \X$ is a finite set of $r \geq 0$ points and $x \in \X-\B$ is a point.
Suppose $G$ is a finite group.

\subsection{The fundamental group $\pi_1(\X-\B, x)$} \label{profinitediff}

The complex curve $\X$ is homeomorphic to the quotient of a polygon with $4g$ sides,
where the quotient is determined by identifying the sides with the consecutive labels  
$\alpha_i$, $\beta_i$, $\alpha_i^{-1}$, $\beta_i^{-1}$ for $1 \leq i \leq g$. 
Also the point $x$ can be identified with a corner of the polygon.
Let $\gamma_i$ be a loop in $\X$, starting at $x$, that circles around the $i$th point of $\B$.
The topological fundamental group $\pi_1(\X-\B, x)$ is generated by the homotopy classes of the loops 
$\alpha_1, \beta_1, \ldots, \alpha_g,\beta_g, \gamma_1, \ldots, \gamma_r$ with the sole 
relation $\prod_{i=1}^g [\alpha_i,\beta_i]\prod_{j=1}^r \gamma_j = 1$.  
This statement about the fundamental group implies the following facts:
\begin{enumerate}
\item[(i)] If $r > 0$, then $\pi_1(\X - \B, x)$ is a free group on $2g+r-1$ generators.
\item[(ii)] The structure of $\pi_1(\X - \B, x)$ depends only on the genus of $\X$ and the 
cardinality of $\B$.
\end{enumerate}

Because there is a bijection between finite quotients of $\pi_1(\X-\B,x)$ and finite Galois covers of $\X$ 
with branch locus in $\B$, one immediately deduces the following:
\begin{enumerate}
\item[(iii)] A finite group $G$ is the Galois group of a cover of $\X$ with branch locus in $\B$ 
if and only if $G$ can be generated by $2g+r-1$ elements.
\item[(iv)] In particular, there are no nontrivial Galois covers of the complex affine line $\Aa_{\CC}$
(i.e., the complex plane is simply connected).
\item[(v)] Given $\X$, $\B$, and $G$, the number of isomorphism classes of Galois covers 
of $\X$ with branch locus in $\B$ and with Galois group $G$ is finite.
\end{enumerate}

\subsection{Ramification of complex covers}\label{Sramcx}

Suppose $\varphi:\Y \to \X$ is a Galois cover of complex curves with branch locus in $\B$
and Galois group $G$.
Consider a point $Q \in \varphi^{-1}(\B)$.
The decomposition group $D_Q$ consists of the automorphisms $\sigma \in {\rm Gal}(\Y/\X)$ such that $\sigma(Q)=Q$. 
The image under $\varphi$ of a loop in $\Y$ around $Q$ will be a loop in $\X$ around $\varphi(Q)$ 
traversed $|D_Q|$ times.
By triangulating $\X$ and $\Y$ appropriately and computing their Euler characteristics, 
one can determine the genus of $\Y$.  This yields some more facts:
\begin{enumerate}
\item[(vi)] The decomposition groups of a Galois cover $\varphi:\Y \to \X$ of complex curves are cyclic.
\item[(vii)] If $g_\X$ is the genus of $\X$, the genus $g_\Y$ of $\Y$ is given by the Riemann-Hurwitz formula to be
$$2g_\Y - 2 = |G|(2g_\X -2) + \sum_{Q \in \varphi^{-1}(\B)} (|D_Q|-1).$$
Thus, $g_\Y$ is determined by $g_\X$, $|G|$, $|\B|$ and the 
orders of the decomposition groups.  
\end{enumerate}

\subsection{Jacobians and torsion points}

The definition of the Jacobian $J_{\X}$ of a complex curve $\X$ can be found in \cite[VIII]{Miranda}.
Recall that $\Omega^1$ is the vector space of holomorphic $1$-forms $\omega$ on $\X$.  
If $\gamma$ is a loop in $\X$, there is a linear functional $\int_\gamma: \Omega^1 \to \CC$.
The value of the integral $\int_\gamma \omega$ depends only on the equivalence class $[\gamma]$ of $\gamma$
in the homology group $H_1(\X, \ZZ)$, which is the abelianization of the fundamental group.
The dual space $(\Omega^1)^*$ is the vector space of linear functionals $\lambda:\Omega^1 \to \CC$.
A {\it period} is a linear functional which equals $\int_{[\gamma]}$ 
for some equivalence class $[\gamma]$ in $H_1(\X, \ZZ)$.
The set $\Lambda$ of periods is a subgroup of $(\Omega^1)^*$.

The Jacobian of $\X$ is $J_{\X}:=(\Omega^1)^*/\Lambda$.  
If $\X$ has genus $g$, then ${\rm dim}(\Omega^1)=g$.
Also $H_1(\X, \ZZ)$ is a $\ZZ$-module of rank $2g$.
Thus $J_{\X} \simeq \CC^g/\Lambda$ is a complex torus of dimension $g$.  
In fact, it is an abelian variety of dimension $g$ \cite[Chapter 6]{GrHa}.

If $\ell$ is a prime, consider the multiplication-by-$\ell$ map $m_\ell$ on $J_{\X}$.
The kernel $J_{\X}[\ell]$ of $m_\ell$ is the subgroup of $\ell$-torsion points of the Jacobian.
As an abelian group, $J_{\X}[\ell] \simeq (1/\ell) \Lambda/\Lambda$, thus:

\begin{enumerate}
\item[(viii)] The subgroup of $\ell$-torsion points of the Jacobian 
satisfies $J_{\X}[\ell] \simeq (\ZZ/\ell)^{2g}$.
In particular, there are $\ell^{2g}$ points of the Jacobian that are $\ell$-torsion points.
\end{enumerate}

\subsection{Transition to characteristic $p>0$}

We will consider covers of curves defined over an algebraically closed field 
$k$ of characteristic $p >0$.
The topological tools used above, such as loops, are meaningless for $k$-curves.
For this reason, new algebraic definitions are needed for objects such as the fundamental group or Jacobian of a $k$-curve. 
Surprisingly, many attributes of fundamental groups and covers will remain the same in characteristic $p$.
Most importantly, Grothendieck proved that fact (iii) holds for finite groups $G$ that are prime-to-$p$ (see Section \ref{SGroth}). 
However, there are some substantial differences between the characteristic $p$ and characteristic $0$ settings.
In particular, we will see that statements (i)-(viii) are each false for covers of $k$-curves.  
In each case, the statement must be revised in characteristic $p$ to cope with the appearance of 
new ramified $p$-group covers and the disappearance of unramified $p$-group covers.

\section{Examples of covers of curves in characteristic $p >0$} \label{Sexample}

Let $k$ be an algebraically closed field of characteristic $p >0$, e.g., $k=\bar{\FF}_p$.
Before developing the theory, we provide some examples of Galois extensions of $k(x)$.  
While the constructions are simple, these examples are crucial for understanding the fundamental group in characteristic $p >0$.

For a field extension $L/K$, let $\Gal(L/K)$ be the set of automorphisms of $L$ fixing every element of $K$.  
A field extension is Galois if and only if $|\Gal(L/K)| = [L:K]$.  

\subsection{Kummer extensions}\label{Skummer}
Let $\ell$ be a prime.
As long as $\ell$ is distinct from $p$, 
Kummer extensions still yield Galois extensions of $k(x)$ with Galois group $\ZZ/\ell$:
\begin{equation} \label{Ekummer}
k(x) \hookrightarrow k(x)[y]/(y^{\ell}-x) \cong k(y).
\end{equation} 
This is an extension of degree $\ell$.  
Since $p^\alpha \equiv 1 \bmod \ell$ for some positive integer $\alpha$, 
there is an $\ell^{\mathrm{th}}$ root of unity $\zeta_\ell \in \FF_{p^\alpha} -\{1\}\subset k$.
Then $\sigma:y \mapsto \zeta_{\ell} y$ is an automorphism of degree $\ell$.
Thus $\Gal(k(y)/k(x))=\langle \sigma \rangle \simeq \ZZ/\ell$ and the extension is Galois.

The only places of $k(x)$ over which $k(y)$ is ramified are $x$ and $1/x$.
To see that $x$ is the only affine place over which the extension is ramified, 
note that $0=\partial(y^{\ell}-x)/\partial y = {\ell}y^{{\ell}-1}$ if and only if $y=0$.
For more information about this example, see \cite[III.7.3]{sti}.

If $\ell =p$, then the polynomial $t^p-1 \equiv (t-1)^p \bmod p$ has only one root in $k$.
So extension (\ref{Ekummer}) has degree $p$, but $\Gal(k(y)/k(x))$ is trivial, and thus
extension (\ref{Ekummer}) is not Galois.

\subsection{Artin-Schreier extensions}
A new equation is needed in order to produce the group $\ZZ/p$ as the Galois group 
of an extension of $k(x)$.
For $f(x) \in k[x]$ with $d=\deg(f(x))$ prime-to-$p$, consider the degree $p$ extension:  
\begin{equation} \label{Eartsch}
k(x) \hookrightarrow L:=k(x)[y]/(y^p-y-f(x)).
\end{equation}
Then $\tau:y \mapsto y +1$ is an automorphism of $L$ of order $p$ because $(y+1)^p \equiv y^p+1 \bmod p$.
Thus $\Gal(L/k(x))=\langle \tau \rangle \simeq \ZZ/p$ and the extension is Galois.

There is no affine place of $k(x)$ over which $L$ is ramified because
$$\partial (y^p-y-f(x))/\partial y =py^{p-1}-1 \equiv -1 \bmod p \neq 0.$$
The only place of $k(x)$ over which $L$ is ramified is the infinite place.
For more information about this example, see \cite[III.7.8]{sti}.

\subsection{New phenomena in characteristic $p>0$}

Artin-Schreier extensions can be used to give counterexamples to some of the facts from Section \ref{Sfalse}
for covers of $k$-curves.
First, consider the affine line $\Aa_k = \PP_k -\infty$, so that $2g+r-1=0$.
The Artin-Schreier extension (\ref{Eartsch}) with equation $y^p-y=f(x)$ yields 
a nontrivial Galois cover $\phi:Y \to \PP_k$ branched only at $\infty$.  
The decomposition group above $\infty$ has order $p$. 
This shows that facts (iii) and (iv) are false for $k$-curves.
Moreover, by changing either the degree or the coefficients of $f(x)$, one sees that these covers occur in infinite families, and thus fact (v) is false for $k$-curves as well.
It turns out that the genus of $Y$ is $(p-1)(d-1)/2$ (see Section \ref{inertia}).
This depends on a new invariant $d={\rm deg}(f(x))$ which shows that fact (vii) is false for $k$-curves.

To construct a counterexample to fact (vi) for $k$-curves, consider a tower of a Kummer and Artin-Schreier extension with equations $x_1^\ell=x$ and $y^p-y=x_1^d$ where $\ell \mid (p-1)$, $p \nmid d$, and $\ell \nmid d$.  
This yields an extension $M/k(x)$ of degree $\ell p$.
Consider the following automorphisms in ${\rm Gal}(M/k(x))$ where $\zeta_\ell$ 
is a primitive $\ell$th root of unity:
$$\tau: x_1 \mapsto x_1, y \mapsto y+1, \ {\rm and} \  \sigma:
x_1 \mapsto \zeta_\ell x_1, y \mapsto \zeta_\ell^d y.$$
Then $\sigma \tau \sigma^{-1}(y)=y + \zeta_\ell^{-d} \not = \tau(y)$.
This shows that the extension is Galois with Galois group $G$ a non-abelian semi-direct product
of the form $\ZZ/p \rtimes \ZZ/\ell$.  The extension is totally ramified above $\infty$ and so the 
decomposition group equals $G$ which is not cyclic.

For a counterexample to fact (viii) for $k$-curves, suppose $p=2$, 
and consider the $k$-curve $E$ defined by the Artin-Schreier equation $y^2-y=x^3$, which is an elliptic curve.  
Then $E$ is supersingular \cite[V, \# 5.7]{sil} and thus the Jacobian of $E$ has no $2$-torsion points other than 
the identity \cite[V, Thm.\ 3.1]{sil}.

It is harder to show, but the same phenomena contradicting facts (iii)-(vii) hold for Galois covers of an arbitrary affine $k$-curve $X - B$ with Galois group $G$ under the basic condition that $p$ divides $|G|$.  The 
same phenomenon contradicting fact (viii) occurs for any smooth projective $k$-curve of positive genus.  
These will be major themes of the next sections.
More theory about Galois covers for curves defined over $k$ is needed 
in order to define the fundamental group of a $k$-curve.
Having done so, we show that facts (i)-(ii) are also false for $k$-curves in Sections \ref{JacCov}, \ref{pro-p}, and \ref{Sanab}.    

\section{Algebraic definitions} \label{Sdefine}

Here we provide the basic definitions required to make sense of covers in arithmetic geometry.  This section is meant to be a reference for the following sections.  The reader may find it easier to skip this section and refer back to it as necessary.  The idea is to mimic the construction of covering spaces in topology and analysis, where $U \to C$ is a covering if locally the inverse function theorem holds.  In the algebraic context, the comparable concept is that of an \'etale or unramified morphism.  

Let $K$ be an algebraically closed field; (the material in this section is valid in any characteristic). Unless stated otherwise, all curves in this section are smooth connected $K$-curves.  Let $X$ be a projective $K$-curve.  The genus of $X$ is the dimension of $H^0(X, \Omega^1)$.
Let $B \subset X$ be a finite (possibly empty) set of points and let $C=X-B$.   

\subsection{Terminology for Galois covers.}

An algebraic field extension $L$ of $F$ is a {\it separable} $F$-algebra if for every element $y \in L$ the minimal polynomial of $y$ over $F$ factors into distinct linear factors in its splitting field. 
The extension is {\it inseparable} otherwise.
For example, extension (\ref{Ekummer}) is purely inseparable when $\ell=p$.
If $R$ is an integral domain and $R\subset S$ is a ring extension, then $S$ is \emph{generically separable} as an $R$-algebra if $\mathrm{frac}(S)$ is a separable $\mathrm{frac}(R)$-algebra.
A morphism of $K$-curves $\phi:Y \to X$ is \emph{generically separable} if $X$ can be covered by affine open subsets $U=\Spec(R)$ such that the ring extension $R\subset\CO(\phi^{-1}(U))$ is generically separable. 
A {\it cover} is a morphism $\phi:Y \to X$ which is finite and generically separable.  

If $\phi:Y \to X$ is a cover, then the \emph{Galois group} $\Gal(Y/X)$ consists of the automorphisms 
$\sigma$ of $Y$ satisfying $\phi \circ \sigma = \phi$.  
If $G$ is a finite group, then a $G$-\emph{Galois cover} is a cover $\phi: Y \to X$ 
\emph{together} with an inclusion $\rho: G \hookrightarrow \Gal(Y/X)$ 
such that $\CO_Y^G = f^*(\CO_X)$ (where the left side denotes the sheaf of $G$-invariants).  
If $Y$ is irreducible, this forces $\rho$ to be an isomorphism.
As $X$ is a smooth curve, this condition is equivalent to saying that $G$ acts simply transitively on a generic geometric fibre of $\phi:Y \to X$, so that $|\Gal(Y/X)| = {\mathrm deg}(f)$.
Given an abstract finite group $G$, there could be many inclusions $\rho$ with this property.
If the inclusion $\rho$ is not fixed then $\phi:Y \to X$ is called a \emph{Galois cover with Galois group $G$}.

For example, extension (\ref{Eartsch}) 
is a Galois cover with group $\ZZ/p$ and extension (\ref{Eartsch}) together with 
the choice of automorphism $\tau:y \mapsto y+1$ is a $\ZZ/p$-Galois extension.

\subsection{Ramification: Wild, tame and $p$-tame}\label{inertia}
Let $\phi:Y \to X$ be a $G$-Galois cover.
The cover $\phi$ is {\it prime-to-$p$} if $|G|$ is prime-to-$p$.

Let $Q$ be a point of $Y$ and let $P = \phi(Q) \in X$.  
The {\it decomposition group} $D_Q$ at $Q$ is the subgroup of $G$ 
consisting of automorphisms that fix the point $Q$.  
The number of points in the fibre $\phi^{-1}(P)$ equals $|G|/|D_Q|$.
The {\it inertia group} $I_Q$ at $Q$ is the subgroup of $D_Q$ that induces the identity automorphism on the residue field at $Q$.  
Since $K$ is algebraically closed, the inertia group equals the decomposition group.  
The cover is {\it ramified} at $Q$ if $I_Q$ is non-trivial, and it is {\it totally ramified} at $Q$ if $I_Q=G$.
The {\it branch locus} of $\phi$ is the set of points $P \in X$ for which there exists a ramified point $Q \in \phi^{-1}(P)$.
The phrase {\it with branch locus in $B$} means that the branch locus is contained in $B$.

When $K$ has characteristic $p>0$, the cover $\phi$ is {\it wildly ramified} at $Q$ if $p$ divides $|I_Q|$ 
and is {\it tame} otherwise.  
The cover $\phi$ is tame if it is tame at all ramification points and is wild otherwise.  
When $K$ has characteristic $0$, a ramified point $Q$ is {\it $p$-tame} 
if $p$ does not divide $|I_Q|$.

For example, if $\ell \not = p$, then equation (\ref{Ekummer}) yields a cover $\phi:Y \to \PP_K$ with branch locus $B=\{0, \infty\}$, 
which is tame because the inertia group is $\ZZ/\ell$ above both branch points.
Equation (\ref{Eartsch}) yields a cover $\phi:Y \to \PP_K$ branched only at $\infty$, which is wild because 
the inertia group is $\ZZ/p$. 

\subsection{The fundamental group}\label{fundgrp}

If $Z \to X$ is a $G$-Galois cover with branch locus in $\B$ and $\pi:G \to H$ is a surjection of finite groups, 
then $Z \to X$ must factor through an $H$-Galois cover $Y \to X$ with branch locus in $B$.  
Consider the set of Galois groups of finite Galois covers  of $X$ with branch locus in $B$.
Consider also the collection of surjections $\pi:G \to H$ 
when a $G$-Galois cover $Z \to X$ with branch locus in $B$ factors through an $H$-Galois cover with branch locus in $B$.
This set of groups and collection of surjections forms an inverse system.
The inverse limit of this system is the algebraic fundamental group $\pi_1(C)$ where $C=X-B$.
The isomorphism class of $\pi_1(C)$ does not depend on the choice of the base point so we eliminate the 
base point from the notation.
A more precise and complete definition of the fundamental group can be found in \cite[Section 2]{Mezard}.

By definition, the finite quotients of $\pi_1(C)$ correspond to finite Galois covers of $C$.
Thus to understand $\pi_1(C)$ one needs to understand:
\begin{enumerate}
\item  What are the finite quotients of $\pi_1(C)$?
\item  How do the finite quotients fit together into an inverse system?
\end{enumerate}
Answering the first question is called ``the inverse Galois problem" for $C$.  The second question is more subtle and is related to embedding problems.  Roughly speaking, question (2) asks:  Given an $H$-Galois cover 
$\psi:V \to C$ and a surjection $G \twoheadrightarrow H$, what $G$-Galois covers of 
$C$ exist which factor through $\psi$?  (See Section \ref{freeness}.)

For an algebraically closed field $K$ of characteristic 0, Grothendieck showed that the fundamental group of a $K$-curve $X-B$ of genus $g$ with $r=|B|$ punctures is isomorphic to the profinite completion of the topological 
fundamental group of a Riemann surface of genus $g$ with $r$ punctures \cite[XIII, Cor.\ 2.12]{Gr}. 
Thus, $\pi_1(X - B)$ is the group obtained by taking the profinite group on generators 
$\alpha_1, \beta_1,...,\alpha_g,\beta_g, \gamma_1,..., \gamma_r$ and imposing the sole 
relation $\prod_{i=1}^g [\alpha_i,\beta_i]\prod_{j=1}^r \gamma_j = 1$.
In particular, if $r > 0$, then $\pi_1(X - B)$ is a free profinite group on $2g+r-1$ generators.  This implies that every group generated by $2g+r-1$ elements is a quotient of $\pi_1(X-B)$. 
Moreover, the freeness implies that:  Given an $H$-Galois cover $\psi: V \to C$, a group $G$ generated by $2g+r-1$ elements, and a surjection 
$G \twoheadrightarrow H$, there exists a $G$-Galois cover of $C$ that factors through $\psi$. 

For example, if ${\rm char}(K)=0$ and $X=\PP_K$ and $B=\{0,\infty\}$, then $\pi_1(X-B)$ 
is the profinite group $\hat{\ZZ}$ on one generator. 
This implies that, for each $\ell \in \NN$, there is exactly one isomorphism class of $\ZZ/\ell$-Galois cover 
$\phi_\ell:Y \to \PP_K$ branched at $B=\{0,\infty\}$.  The Kummer Equation (\ref{Ekummer}) in Section \ref{Skummer} is an equation representing this isomorphism class.  
If $\ell_1 \mid \ell_2$, then $\phi_{\ell_2}$ factors through $\phi_{\ell_1}$.

In summary, for a curve defined over an algebraically closed field $K$ of characteristic $0$, the fundamental group is finitely generated as a profinite group; in particular, this implies that the answer to question (2) above is completely determined by the answer to question (1). This is because a finitely generated profinite group is determined by its finite quotients \cite[Prop.\ 15.4]{FJ}. 

The fundamental group of a curve defined over an algebraically closed field $k$ of characteristic $p>0$ 
is known in only two cases: (1) when $X$ is the projective line $\PP_k$ and $B=\emptyset$ then $\pi_1(X)$ is trivial; and (2) when $X$ is an elliptic curve $E$ and $B = \emptyset$ then $\pi_1(E)$ is a finitely generated abelian profinite group, 
(see subsection \ref{JacCov}).
Section \ref{Stheory} contains some of the major results obtained about the fundamental group and its finite quotients in characteristic $p >0$.   In particular, when $B \not = \emptyset$, the fundamental group is not finitely generated as a profinite group and is not determined by its finite quotients.  Thus understanding how the groups fit together, as in question (2), is essential to determining the profinite group structure of the fundamental group.

\subsection{Translation into field theory} The material in Section \ref{fundgrp} can be reinterpreted in terms of field extensions as follows.  The function field $K(C)$ of $C$ is the same as that of $X$.   A separable closure $K(C)^{\mathrm sep}$ is an infinite Galois extension of $K(C)$ whose Galois group $\Gal_{K(C)}$ is called the absolute Galois group of $C$.
  

Given a $G$-Galois extension $L/K(C)$,   
consider an open cover of $C$ by affine opens $U_i = \Spec(R_i)$ and let $S_i$ be the integral closure of $R_i$ in $L$ and let $V_i=\Spec(R_i)$.  The affine opens $V_i$ cover a curve $V$ and there is a cover $V \to C$.  There is a $G$-Galois cover $\phi:Y \to X$ where $Y$ is the projective closure of $V$.
However, any point of $X$ could be a branch point of $\phi$
so this may not correspond to a surjection of $\pi_1(C)$ onto $G$.

To remedy this, one instead considers the maximal Galois extension $K(C)_{un,B}$ of $K(C)$ unramified outside of the set of places in $K(C)$ for points in $B$.  Then there is a bijection between surjections $\pi: \Gal(K(C)_{un,B}/K(C) \twoheadrightarrow G$ and $G$-Galois covers $\phi: Y\to X$ with branch locus in $B$.  Furthermore, 
$\pi$ factors through a surjection $\pi': \Gal(K(C)_{un,B}/K(C) \twoheadrightarrow \Gamma$ if and only if the 
$G$-Galois cover $\phi$ can be dominated by a $\Gamma$-Galois cover $\phi'$ with branch locus in $B$.   
Thus, $\Gal(K(C)_{un,B}/K(C) =\pi_1(C)$.  

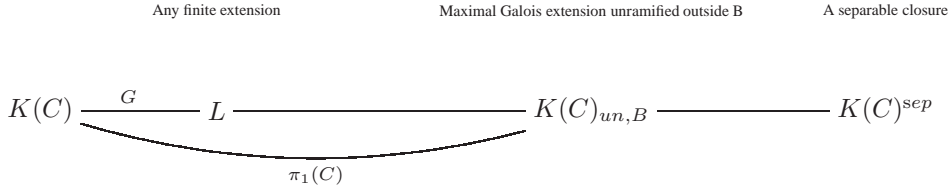
\[
\xymatrix {
&\mbox{\tiny Any finite extension}&&\mbox{\tiny  Maximal Galois extension unramified outside B} &\mbox{\tiny A separable closure}\\
K(C)\ar@{-}[r]^G\ar@/_1.5pc/@{-}[rrr]_{\pi_1(C)}&L\ar@{-}[rr] &&K(C)_{un,B}\ar@{-}[r]&K(C)^{{\mathrm sep}}
}
\]

\subsection{Higher ramification groups} 

There is extra ramification information at a wildly ramified point $Q$, including a filtration of $I_Q$ 
called the filtration of higher ramification groups, \cite[IV]{Se:lf}.
If $\phi:Y \to X$ is ramified at $Q$, consider the complete local ring $\hat{\mathcal O}_Q$ of functions at $Q$ and the valuation function $\nu_Q$.  For any integer $i \geq -1$ the \emph{$i$th ramification group} at $Q$ is $$I_i(Q)=\{\sigma \in D_Q | \nu_{Q}(\sigma(z)-z) \geq i+1 ,\forall z \in \hat{\mathcal O}_Q\}.$$
The decomposition group at $Q$ is $I_{-1}(Q)$ and the inertia group is $I_{0}(Q)$.  
The inertia at a wildly ramified point is usually not cyclic though it is always cyclic-by-$p$,
in that it has a normal Sylow $p$-subgroup $I_1(Q)$ and the quotient $I_Q/I_1(Q)$ is cyclic and prime-to-$p$.

The genus of $Y$ now depends on the ramification filtration.
The Riemann-Hurwitz formula states that $2g_Y-2=|G|(2g_X-2) + {\rm Ram}$ where ${\rm Ram}$ is the sum of the 
degrees of the different at each ramified point $Q$.  
The degree of the different at $Q$ equals $\sum_{i=0}^{\infty} (|I_i(Q)|-1)$.
If ${\rm char}(K)=0$, then $\rm Ram = \sum_{Q \in \phi^{-1}(B)} (|I_Q|-1)$ which recovers the equation in part (vii) of Section \ref{Sramcx}.  If ${\rm char}(K)>0$, then $\rm Ram \geq \sum_{Q \in \phi^{-1}(B)} (|I_Q|-1)$.  From this, one sees that the genus of $Y$ can grow ``more quickly" in positive characteristic than in characteristic zero.

For the Artin-Schreier extension $y^p-y=f(x)$ 
in Equation \ref{Eartsch} where ${\rm deg}(f(x))=d$ and $p \nmid d$, if $Q$ is the point above $\infty$ then $I_i(Q)=\ZZ/p$ if $0 \leq i \leq d$ and $I_i(Q)=\{0\}$ if $i >d$, \cite[III.7.8(c)]{sti}. 
By the Riemann-Hurwitz formula, $g_Y=(p-1)(d-1)/2$.

\subsection{The Jacobian and torsion points} \label{Sjacobian}

Let $X$ be a smooth projective $K$-curve of genus $g$. 
A divisor on $X$ is a formal sum $\sum_{P \in X} n_PP$ where $n_P \in \ZZ$ and $n_P = 0$ for all but finitely many $P \in X$.   
The degree of a divisor $D=\sum_{i=1}^r n_iP_i$ is $\sum_{i=1}^r n_i$ and ${\rm Div}^0(X)$ denotes the abelian group of all divisors of $X$ of degree $0$.
Given a non-zero element $f$ in the function field $K(X)$ of $X$, there is 
a divisor $\dv(f)=\sum_{P \in X} {\rm ord}_P(f)P$.
A divisor $D$ is {\it principal} if $D=\dv(f)$ for some function $f \in K(X)$.
Every principal divisor has degree zero.
Let ${\rm Prin}(X)$ be the set of all principal divisors of $X$.
The sets ${\rm Div}^0(X)$ and ${\rm Prin}(X)$ are abelian groups under addition and ${\rm Prin}(X) \subset {\rm Div}^0(X)$.
The algebraic definition of the Jacobian $J_X$ of $X$ is $J_X:={\rm Div}^0(X)/{\rm Prin}(X)$.
This is an abelian group which is naturally isomorphic to the $K$-points  of an abelian variety of dimension $g$ over $K$, which we also denote $J_X$.

For a prime $\ell$, consider the multiplication-by-$\ell$ morphism $m_\ell$ on $J_X$.
The $\ell$-torsion $J_X[\ell]$ of the Jacobian is the kernel of $m_\ell$.
The $K$-points of $J_X[\ell]$ can be identified with the set 
$$\{[D]  \in J_X ~|~\rm{there ~exists} ~ f \rm{~such~that} ~ \ell D = \dv(f)\}.$$
If $\ell \not = p$, then $m_\ell$ is separable of degree $\ell^{2g}$.
Thus $J_X[\ell]\simeq (\ZZ/\ell)^{2g}$ \cite[pg. 64]{mumford}.

For example, suppose ${\rm char}(K) \not = 2$ and $Y$ is a hyperelliptic $K$-curve with equation $y^2=\prod_{i=1}^{2g+1}(x-b_i)$.  Let $Q_\infty$ be the point at infinity of $Y$ and let $Q_i$ 
be the point $(x,y)=(b_i, 0)$ for $1 \leq i \leq 2g+1$.  
The divisors $D_i=Q_i-Q_\infty$ are 2-torsion points of $J_Y$ since $2D_i=\dv(x-b_i)$.
There is a relation $0=\dv(y)=\sum_{i=1}^{2g+1} D_i$ in $J_Y$. 
There is no nontrivial linear relation $\sum_{i=1}^{2g} a_i D_i =0$ with $a_i \in \{0,1\}$
because $\prod_{i=1}^{2g} (x-b_i)^{a_i}$ is not a square in $k[x,y]/(y^2 - \prod_{i=1}^{2g+1} (x-b_i))$.
Thus the set $\{D_1, \ldots, D_{2g}\}$ is linearly independent and hence is a basis for $J_Y[2]$. 

In contrast, if $p={\rm char}(K)$,
the multiplication-by-$p$ morphism $m_p$ factors as the composition of the Frobenius morphism, which is inseparable of degree $p^g$
and the Verschiebung morphism which is separable of degree $p^g$.
This implies that $J_X[p]$ is a group scheme of rank $p^{2g}$.
The number of points in $J_X[p](K)$ 
equals $p^s$ for some integer $s$ such that $0 \leq s \leq g$.
Here $s$ is called the {\it $p$-rank} of $X$. 

For example, an elliptic curve defined over an algebraically closed field of characteristic $p$ 
can be either ordinary ($s=1$) or supersingular ($s=0$).
If $f(x)=x(x-1)(x-\lambda)$, the elliptic curve $y^2=f(x)$ is supersingular if and only if 
$\lambda$ is a root of the coefficient of $x^{p-1}$ in $f(x)^{(p-1)/2}$ \cite[V, Thm.\ 4.1]{sil}.
It can be computationally difficult to determine the $p$-rank of a curve of higher genus.
For Jacobians of hyperelliptic curves, an algorithm to compute the $p$-rank can be found in \cite{Y}.
Another situation where the $p$-rank can be computed is when $\phi:Y \to X$ is a Galois cover 
whose Galois group $G$ is a $p$-group.  If $|G|=p^a$, 
then the Deuring-Shararevich formula \cite[Cor. 1.8]{DS} states that
$$s_Y-1=p^a(s_X-1)+\sum_{Q \in Y}(|I_Q|-1).$$

\subsection{Unramified covers and the Jacobian}\label{JacCov}

In this section, we describe a connection between the $\ell$-torsion points of the Jacobian of $X$ and unramified $\ZZ/\ell$-Galois covers of $X$.

For some intuition about this connection, consider the example of an elliptic curve $E$ over $K$.
If $\ell$ is prime-to-$p$, then the multiplication-by-$\ell$ morphism $m_\ell: E \to E$ is a separable 
cover of degree $\ell^2$ \cite[III, Cor.\ 5.4]{sil}.  
Suppose $Q \in E$ is an $\ell$-torsion point.  If $R\in E$, then $m_\ell(R+Q) =m_\ell(R)$.
Thus there is an automorphism $\sigma_Q$ of $E$ of order $\ell$ defined by $\sigma_Q(R)=R+Q$
and $m_\ell \circ \sigma_Q = m_\ell$.
In other words, $m_\ell$ is a Galois cover, whose Galois group can be identified with 
$J_E[\ell]$.   
After choosing a basis for $J_E[\ell]$, then $m_\ell$ is a $(\ZZ/\ell)^2$-Galois cover.

Continuing this example, suppose $Q'$ is a point of order $\ell$ on the Jacobian of $E$.  
Then $Q'$ can be canonically identified with a point $Q$ of order $\ell$ on $E$, \cite[X, Thm.\ 3.8]{sil}.
Consider the subgroup $H_Q=\langle \sigma_Q \rangle \subset J_E[\ell]$.
Note that $J_E[\ell]/H_Q$ is a cyclic group of order $\ell$.
Let $E_Q$ be the quotient of $E$ by $H_Q$.
The quotient cover $E_Q \to E$ is a $\ZZ/\ell$-Galois cover, which is unramified by the Riemann-Hurwitz theorem.  
To summarize, every $\ell$-torsion point on the Jacobian of $E$ yields an unramified $\ZZ/\ell$-Galois cover of $E$.
Using this, one can prove that
$$\pi_1(E) \simeq \hat \ZZ_p^s \times \prod_{\ell \not = p} (\hat{\ZZ}_\ell \times  \hat{\ZZ}_\ell)$$
where $s=1$ if E is ordinary, and $s=0$ if E is supersingular \cite[IV, exercise 4.8]{Hart}.
This gives a counterexample to fact (ii) from Section \ref{Sfalse} 
for a projective curve defined over an algebraically closed field of characteristic $p>0$.

For a projective $K$-curve $X$ of higher genus, the bijection between $\ell$-torsion points of $J_X$ 
and unramified $\ZZ/\ell$-Galois covers of $X$ is harder to construct.  
As in \cite[III, Section 4]{milne}, 
one defines $\pi^1(X,\ZZ/\ell)$ to be the set of isomorphism classes of unramified 
$\ZZ/\ell$-Galois covers of $X$.     
By \cite[Remarks following Section III, Prop 4.11]{milne}, 
for $\ell$ prime to ${\rm char}(K)$, 
the group $\pi^1(X,\ZZ/\ell) \cong H^1(X,\ZZ/\ell)$ is isomorphic to $J_X[\ell](K)$. 
Similarly, 
$\pi^1(X,\ZZ/p) \cong H^1(X,\ZZ/p) \cong J_X[p](K)$ for $p={\rm char}(K)$ by 
\cite[Remarks following Section III, Prop 4.13]{milne}. 
Thus the $p$-rank equals the maximum rank of a $p$-group
which occurs as the Galois group of an unramified cover of $X$ \cite[Cor.\ 4.18]{milne}.

\section{Major results} \label{Stheory}

Let $k$ be an algebraically closed field of characteristic $p >0$. 
Let $X$ be a smooth connected projective $k$-curve of genus $g$.
Let $B \subset X$ be a finite subset of $r$ points.  

\subsection{The prime-to-$p$ fundamental group}\label{SGroth}

The main point of this section is Grothendieck's result that the prime-to-$p$ groups that occur for Galois covers 
in characteristic $p$ are exactly the same as those that occur in characteristic $0$, 
namely those generated by $2g+r-1$ elements.
Recall the definitions of prime-to-$p$, tame, and $p$-tame covers from Section \ref{inertia}.
The {\it prime-to-$p$ fundamental group} $\pi_1^{p'}(X-B)$ is the inverse limit of the
Galois groups of finite Galois prime-to-p covers of $X$ with branch locus in $B$.
The {\it tame fundmental group} $\pi_1^t(X-B)$ is the inverse limit of the
Galois groups of finite Galois covers of $X$ with branch locus in $B$ and tamely ramified over $B$.  
The result also shows that $\pi^t(X-B)$ and $\pi_1^{p'}(X-B)$ are finitely generated as profinite groups and, as such, 
are determined by their finite quotients \cite[Prop.\ 15.4]{FJ}. 

The basic idea behind Grothendieck's proof is that tame covers in characteristic $p$ lift to $p$-tame covers in characteristic $0$.  More precisely, let $A$ be a complete local ring with residue field $k$.
By \cite[Exp.\ V, Cor.\ 7.4]{Gr}, 
there exists a smooth projective $A$-curve $\X_A$ such that the closed fibre is $X$.
In particular, taking $A$ to be a complete discrete valuation ring of mixed characteristic with residue field $k$, then the generic fibre $\X$ of $\X_A$ is a lift of $X$ to characteristic $0$.
Let $\B_A$ be a set of horizontal sections specializing to $B$ and let $\B$ be its generic fibre. 
The subset $\B \subset \X$ is a lift of $B \subset X$ to characteristic $0$.
Given a tame $G$-Galois cover $\phi:Y \to X$ with branch locus in $B$, there exists a $G$-Galois cover $\varphi_A:\Y_A \to \X_A$
with branch locus in $\B_A$ whose special fibre is isomorphic to $\phi$.
The generic fibre is a $p$-tame $G$-Galois cover $\varphi: \Y \to \X$ with branch locus in $\B$ in characteristic zero. 

Let $\pi_1^{p-\mathrm{tame}}(\X -\B)$ be the inverse limit of the system of finite groups that occur as Galois groups of $p$-tame covers of $\X$ with branch locus in $\B$. 
Also, consider the prime-to-$p$ fundamental group $\pi_1^{p'}(\X-\B)$.
The previous paragraph summarizes the main ideas of the proof of the following result.

\begin{theorem} \cite[XIII, Cor.\ 2.12]{Gr}\label{GR}
With notation as above, $$\pi_1^{p-\mathrm{tame}}(\X -\B) \twoheadrightarrow \pi_1^t(X -B)$$ and
$$\pi_1^{p'}(\X -\B) \simeq \pi_1^{p'}(X -B).$$
\end{theorem}

In particular, taking $B=\emptyset$ in Theorem \ref{GR} yields a surjection
$\pi_1(\X) \twoheadrightarrow \pi_1(X)$ and thus $\pi_1(X)$ is finitely generated as a profinite group.
The kernel of this homomorphism is not well-understood.

\subsection{The pro-$p$ fundamental group}\label{pro-p}

Unlike the case for prime-to-$p$ fundamental groups, the structure of the pro-$p$ fundamental group of a curve
changes significantly in characteristic $p$, and depends crucially on whether the curve is projective or affine. 
Let $\pi_1^{p}(X-B)$ be the inverse limit of the system of finite $p$-groups that occur as Galois groups of 
covers of $X$ with branch locus in $B$.

\begin{theorem} \cite{shaf}, \cite[Thm.\ 1.9]{DS} \label{ds}
If $X$ is a projective $k$-curve with $p$-rank $s_X$, then
the pro-$p$ fundamental group $\pi_1^p(X)$ is a free pro-$p$ group on $s_X$ generators.  
\end{theorem}

\begin{proof} (Outline following \cite{DS}.)  If $\pi$ is a pro-$p$ group then (1) the minimal number of generators of $\pi$ is equal to ${\rm dim}_{\FF_p} {\rm Hom}(G, \ZZ/p)$; and (2) $\pi$ is free if and only if $H^2(G,\ZZ/p) =0$ (\cite[Thm.\ 12 and Cor.\ 2 to Prop.\ 23]{shatz}).  Thus it suffices to show that ${\rm dim}_{\FF_p} {\rm Hom}(\pi_1(X), \ZZ/p) = s_X$ and $H^2(\pi_1(X),\ZZ/p) =0$. 
Item (1) is discussed in Section \ref{JacCov}.  
Item (2) follows from the fact that the $p$-cohomological dimension of $X$ is less than two \cite[IX, 3.5]{SGA4}.  
\end{proof}

\begin{theorem} \cite{shaf}\label{shaf}
If $C$ is an affine $k$-curve, then the pro-$p$ fundamental group $\pi_1^p(C)$ is infinitely generated.
\end{theorem}

The proof of Theorem \ref{shaf} relies on cohomological arguments and Artin-Schreier theory.  
Since the pro-$p$ fundamental group is a quotient of the fundamental group, Theorem \ref{shaf} 
implies that the fundamental group $\pi_1(C)$ is infinitely, not finitely, generated 
when $C$ is an affine $k$-curve.
Thus the fundamental group of an affine $k$-curve is not determined by its finite quotients  \cite[Prop.\ 15.4]{FJ}.
Moreover, as the $p$ and prime-to-$p$ parts are respectively infinitely and finitely generated as profinite groups, this shows $\pi_1(C)$ is not free and thus fact (i) is false.

\subsection{Abhyankar's Conjecture} \label{Sabhy}

Remarkably, the finite quotients of the fundamental group of every affine $k$-curve are known, even if it is not clear how they fit together.
The next result shows that a finite group $G$ occurs as a Galois group of a cover of $X$ with branch locus in $B$ 
if and only if the maximal prime-to-$p$ quotient of $G$ occurs as a Galois group of a cover of an arbitrary 
genus $g$ curve with $r$ branch points in characteristic $0$.

\begin{theorem}\label{AC}\cite{ab, har, ray}
Let $X$ be a projective $k$-curve of genus $g$ 
and let $B \subset X$ be a finite set of cardinality $r>0$.  
A finite group $G$ is a quotient of 
$\pi_1(X-B)$ if and only if every prime-to-$p$ quotient of $G$ can be generated by $2g+r-1$ elements.
\end{theorem}

Theorem \ref{AC} was conjectured by Abhyankar in 1957, based on his experience working with covers of the affine line.
The proof was completed in 1993 by Raynaud and Harbater.
It is worth noting that the collection of groups which occur as Galois groups of affine $k$-curves is vast.
Not only does every finite $p$-group occur as a quotient of the fundamental group of every affine $k$-curve, 
but every finite simple group of order divisible by $p$ does as well.  
An immediate consequence of Theorem \ref{AC} is the following corollary.

\begin{corollary}
A finite group is a quotient of $\pi_1(\Aa_k)$ if and only if it has no nontrivial prime-to-$p$ quotient.
\end{corollary}

A finite group with no nontrivial prime-to-$p$ quotient is called {\it quasi-$p$}.
Equivalently, a finite group is quasi-$p$ if it is generated by its elements of $p$-power order, or by its 
Sylow $p$-subgroups.

A first step in proving Theorem \ref{AC} is the following result of Serre.
In particular, it shows that every finite $p$-group occurs as a Galois group of a cover of the affine line.

\begin{theorem}\cite[Thm.\ 1]{serre} \label{Tserre}
Suppose $\tilde{G}$ is a finite quasi-$p$ group and $N \subset \tilde{G}$ is a normal subgroup which is solvable.
Let $G=\tilde{G}/N$.
If $G$ is a quotient of $\pi_1(\Aa_k)$ then so is $\tilde{G}$.
\end{theorem}

\begin{proof} (Outline following \cite{serre}.)
Let $\pi$ denote $\pi_1(\Aa_k)$.
By hypothesis, there exists a surjection $\psi: \pi \twoheadrightarrow G$.
Since $N$ is solvable, using representation theory, one can reduce to the case where $N$ is an elementary abelian group and the action of $G$ on $N$ is irreducible.  Then $\tilde{G}$ is an extension of $G$ by $N$, so this yields a cohomology class $e \in H^2(\tilde{G},N)$.  There are two cases: (a) when $e \neq 0$  and (b) when $e=0$.

In case (a) the theorem follows by using the fact that $\pi$ has cohomological dimension at most $1$.  This allows one to lift the surjection $\psi$ to a homomorphism 
$\tilde{\psi}: \pi \to \tilde{G}$, and thus it suffices to show that $\tilde{\psi}$ is surjective.  
The group $H = {\rm Im}(\tilde{\psi})$ is a subgroup of $\tilde{G}$ such that $N \cdot H = \tilde{G}$.  Because $e \neq 0$, $N \cap H$ is a non-trivial sub-$\tilde{G}/H$-module of $N$.  By the irreducibility of $N$, then $N \cap H = N$ and thus $H = \tilde{G}$ and $\tilde{\psi}$ is surjective.  

In case (b) when $e =0$, then $\tilde{G}$ is a semi-direct product.  A surjection $\psi:\pi \to G$
induces a $\pi$-module structure $N_\psi$ on $N$.  Then $\psi$ factors through 
a surjection $\tilde{\psi}:\pi \to \tilde{G}$ if and only if the \'etale cohomology group 
$H^1(\pi, N_\psi)$ is strictly larger than the cohomology group $H^1(G, N)$.
When $N$ is an elementary abelian $p$-group, then $H^1(\pi, N_\psi)$ has infinite dimension, 
which completes the proof.
When $N$ is an elementary abelian $\ell$-group for a prime $\ell \not = p$,
Serre uses the Grothendieck-Ogg-Shafarevich formula 
to calculate the dimension $H^1(\pi, N_\psi)$ in terms of the filtration of higher ramification groups. 
It is possible that the dimension is not large enough, 
in which case it is necessary to change $\psi$ to complete the proof.
\end{proof}

\begin{proof} (Outline of proof of Theorem \ref{AC} following \cite{har, ray}.)
The forward direction of Abhyankar's Conjecture follows from Theorem \ref{GR}.
Here is a sketch of the converse in the case that $X=\PP_k$ and $B=\{\infty\}$.
Suppose $G$ is a finite quasi-$p$ group.
The proof proceeds by induction on the order of $G$. 
By Theorem \ref{Tserre}, one can assume that $G$ has no normal $p$-group subgroup.
Let $S$ be a fixed Sylow $p$-subgroup of $G$.

{\bf Case A:} $G$ is generated by two proper quasi-$p$ subgroups $G_1$ and $G_2$
satisfying the extra condition that $G_i \cap S$ is a Sylow $p$-subgroup of $G_i$.
Inductively, one can suppose that $G_1$ and $G_2$ are each a quotient of $\pi_1(\Aa_k)$.
Then there exists a $G_i$-Galois cover $\phi_i:Y_i \to \PP_k$ branched only at $\infty$ for $i=1,2$.
Harbater's contribution to the proof was to develop a theory of formal patching and use it to patch the 
two covers $\phi_1$ and $\phi_2$ together.  
In this way, a $G$-Galois cover $\phi:Y \to \PP_k$ branched only
at $\infty$ is produced.  The basic idea is to build a $k[[t]]$-curve $W$ whose generic fibre is a projective line and whose special fibre is a chain of two projective lines $W_1$ and $W_2$ intersecting in exactly one ordinary double point.
One can construct a $G$-Galois cover of the special fibre such that its restriction to $W_i$ is 
${\rm Ind}_{G_i}^G (\phi_i)$.  The condition on the Sylow $p$-subgroups allows one to do this compatibly near the 
ordinary double point.  One shows that the cover can be deformed over $k[[t]]$ 
near the ordinary double point. 
Using formal patching, one can deform the cover of projective 
curves over $k[[t]]$.  (A similar technique with formal patching allows one to reduce the 
proof of the converse direction of Abhyankar's Conjecture
for an arbitrary affine $k$-curve to the case of the affine line.)

{\bf Case B:} $G$ is a finite quasi-$p$ group, with no normal $p$-subgroup, not generated by 
proper quasi-$p$ subgroups satisfying the above condition on their Sylow $p$-subgroups.  While the conditions seem
awkward, this turns out to be \emph{exactly} the case that can be handled by Raynaud's analysis of 
semi-stable reduction of covers.  The idea is to consider a $G$-Galois cover $\varphi$ 
of the projective line over ${\rm frac}(W(k))$ 
whose inertia groups are $p$-groups.  The cover $\varphi$ exists by Riemann's Existence Theorem 
since $G$ is quasi-$p$ and thus can be generated by elements of $p$-power order.    
The special fibre $\phi_s$ of the semi-stable reduction of $\varphi$ is a $G$-Galois cover 
of a tree of projective lines.  
The cover $\phi_s$ is inseparable exactly over the interior components of the tree.  Over each 
terminal component of the tree, $\phi_s$ is ramified only over the node $\eta$ 
at which the terminal component intersects the interior of the tree.  
This yields a cover of a projective line branched only at one point $\eta$.  The unusual 
group theoretic conditions insure that there is at least one terminal component over which the restriction of 
$\phi_s$ is connected.  
Thus there exists a $G$-Galois cover $\phi:Y \to \PP_k$ branched only at one point.
\end{proof}

\subsection{Anabelian results} \label{Sanab}

Theorem \ref{AC} implies that the fundamental group of an affine $k$-curve 
is an infinitely generated profinite group.
An interesting consequence of this result is that, for an affine $k$-curve $X-B$, 
the structure of $\pi_1(X-B)$ is not determined by its finite quotients.
This is different than the situation for projective $k$-curves or for curves 
defined over algebraically closed fields of characteristic $0$.

Grothendieck's anabelian conjecture predicts that the isomorphism class of a hyperbolic curve
(a smooth curve whose geometric fundamental group is nonabelian) defined over a number field 
is determined by the structure of its arithmetic fundamental group.
This conjecture has been settled in large part by Mochizuki \cite{Mo99}.
Specifically, let $K$ be a field that can be embedded in a finitely generated field extension of ${\QQ}_p$.
Let $X$ be a smooth $K$-variety and let $Y$ be a hyperbolic $K$-curve. 
Then there is a natural bijection between the set of dominant
$K$-morphisms $X \to Y$ and the set of (conjugacy classes of) open homomorphisms 
$\pi_1(X) \to \pi_1(Y)$ compatible with the action of the absolute Galois group $G_K$.
As an application, he proves a birational version of the anabelian conjecture for function fields
of arbitrary dimension over $K$.
The result builds upon the work of others, especially Tamagawa \cite{Ta97}, 
who introduced a characteristic $p$ version of the Grothendieck anabelian conjecture in characteristic $p$ 
and proved it for affine curves defined over finite fields.

Most anabelian theorems are too technical to include here, but
as a very special case, consider the following result.

\begin{theorem}\cite[Cor.\ 1.8, Thm.\ 1.9]{Ta} \label{Ttama}
Suppose $X$ is a smooth projective $k$-curve and $B \subset X$ is a finite set of points.
The fundamental group $\pi_1(X-B)$ determines the genus of $X$, the cardinality of $B$, 
and the $p$-rank of $X$. 
\end{theorem}

\begin{proof} (Outline following \cite{Ta}.)
The fundamental group $\pi_1(X-B)$ determines the pro-$p$ and the 
prime-to-$p$ fundamental groups.    
By taking the quotient by the commutator, one can determine the abelianization $\pi_1^{\rm ab}(X-B)$ 
of $\pi_1(X-B)$.
If $H$ is an open subgroup of $\pi_1(X-B)$, let $\phi_H: U_H \to X-B$ denote the corresponding cover. 
The degree of $\phi_H$ equals the index $[\pi_1(X-B):H]$ and the Galois group of 
$\phi_H$ is abelian if and only if $H$ contains the commutator. 
Thus these attributes of $\phi_H$ are determined by the fundamental group.

Let $g_X$ be the genus of $X$, let $r_X=\#B$, and let $s_X$ be the $p$-rank of $X$.  
By Theorems \ref{ds} and \ref{shaf}, the structure of $\pi_1^p(X-B)$
determines whether $r_X=0$.  
If $r_X=0$, then the pro-$p$ fundamental group determines the $p$-rank $s_X$ by Theorem \ref{ds}.  
Consider the rank of the maximal elementary abelian $\ell$-group quotient of the abelianization 
$\pi_1^{\rm ab}(X)$ for a prime $\ell \not = p$.
By Section \ref{JacCov}, this is the rank of $J_X[\ell]$, which equals $2g_X$ by Section \ref{Sjacobian}.

Now suppose $r_X >0$.
By Theorem \ref{GR}, the fundamental group determines the quantity $2g_X -2 +r_X$.
Suppose $H$ is an open subgroup of $\pi_1(X-B)$.
The corresponding cover $\phi_H: U_H \to X-B$ has degree equal to the index $[\pi_1(X-B):H]$.  
Note that $H$ is the fundamental group of $\pi_1(U_H)$ and that 
$U_H$ is an open subset of a smooth projective $k$-curve $Y_H$ of genus $g_H$.  
Let $r_H=\# (Y_H-U_H)$.  
It follows that the quantity $2g_H-2+r_H$ can be determined for any open subgroup $H \subset \pi_1(X-B)$.    
Applying the Riemann-Hurwitz formula, one can determine whether $\phi_H$ is wildly ramified.  Thus 
$\pi_1(X-B)$ determines the tame fundamental group.
Now $s_X$ equals the rank of the maximal pro-$p$ quotient of $\pi_1^t(X-B)$.
Similarly, one can determine the $p$-rank $s_H$ of $Y_H$ for any open subgroup $H \subset \pi_1(X-B)$.

Continuing in the case when $r_X >0$, the Deuring-Shafarevich formula 
states that $s_H-1+r_H=p(s_X-1+r_X)$ for any open normal subgroup $H \subset \pi_1(X-B)$ of index $p$.
Thus one can determine the quantity $pr_X-r_H$ which is a multiple of $p-1$. 
By the Riemann-Roch Theorem, there is a function $f \in k(X)$
whose set of poles is $B$ and such that $p \nmid {\rm ord}_b(f)$ for each $b \in B$.
The Artin-Schreier equation $y^p-y=f$ determines a Galois degree $p$ cover $\phi:Y \to X$ which is totally 
ramified above each $b \in B$.  
If $H \subset \pi_1(X-B)$ is the corresponding open normal subgroup, then $r_{H}=r_X$ and 
$pr_X-r_H$ equals $(p-1)r_X$.
Since no smaller value of $pr_X-r_H$ can occur for any 
open normal subgroup $H \subset \pi_1(X-B)$ of index $p$,
the fundamental group determines the value of $r_X$ and thus of $g_X$ as well.
\end{proof}

Theorem \ref{Ttama} finally provides a counterexample to fact (ii) for affine $k$-curves.  
An element $b \in k-\{0,1\}$ is called supersingular if the elliptic curve $E$ with equation 
$y^2=x(x-1)(x-b)$ is supersingular.
By a result of Igusa \cite[V, Thm.\ 4.1(c)]{sil}, there are $(p-1)/2$ supersingular values of $b$
if $p$ is odd.

\begin{proposition} \label{Psupsing} Suppose $p$ is an odd prime.  Let $X=\PP_k$ and $B=\{0,1,b,\infty\}$.
The structure of the fundamental group $\pi_1(X-B)$ depends on whether $b$ is a supersingular value.
\end{proposition}

Note that the finite quotients of $\pi_1(X-B)$ in Proposition \ref{Psupsing} 
do not depend on whether $b$ is a supersingular value by Theorem \ref{AC}.

\begin{proof}
The branch locus of a degree $2$ cover of $\PP_k$ has even cardinality, so there are $7$ subgroups of $\pi_1(X-B)$ of index two.
Of these, exactly one corresponds to a degree two cover $\phi: Y \to \PP_k$ with branch locus in $B$ 
such that $Y$ has positive genus, namely the cover with equation $y^2=x(x-1)(x-b)$.  
By Theorem \ref{Ttama}, one can distinguish the corresponding subgroup $\pi' \subset \pi_1(X-B)$ of index two.
Now $\pi'$ is the fundamental group of $Y-\phi^{-1}(B)$.
Applying Theorem \ref{Ttama} again, one can determine the $p$-rank of $Y$, and in particular determine 
whether $Y$ is ordinary or supersingular.
This determines whether $b$ is a supersingular value.
\end{proof}

For more results along the lines of Proposition \ref{Psupsing}, see \cite{Bo01}, \cite{Ta}.  

\subsection{Freeness results and embedding problems}\label{freeness}

To understand the structure of the fundamental group of a $k$-curve,
it is crucial to understand how its finite quotients fit together.
Harbater and Pop independently proved the following result.

\begin{theorem} \label{Tfree} \cite[Thm.\ 3.5]{har95} \cite[Thm.\ B]{pop}  
The absolute Galois group $G_{k(X)}$ of the function field of a projective $k$-curve $X$ is free of rank 
${\rm card}(k)$.  
\end{theorem}

Before describing the proof of Theorem \ref{Tfree},
it is useful to introduce the terminology of embedding problems.  
Suppose there exists a $G$-Galois cover $\phi:Y \to X$ with branch locus in $B$ corresponding to a surjection $\beta: \pi_1(X-B) \twoheadrightarrow G$ and a surjection $\alpha: \Gamma \twoheadrightarrow G$ of finite groups. 
By Galois theory, a surjection $\lambda: \pi_1(X-B) \twoheadrightarrow \Gamma$ where $\alpha \circ \lambda = \beta$ corresponds to a (connected) $\Gamma$-Galois cover $Z \to X$ with branch locus in $B$ that dominates $\phi$. 

Given a group $\Pi$, a pair of surjections $(\beta: \Pi \twoheadrightarrow G, \alpha: \Gamma \twoheadrightarrow G)$  where $G$ and $\Gamma$ are finite groups is a {\it finite embedding problem} for $\Pi$.
In the case where $\alpha$ has a splitting $s: G \to \Gamma$, the pair is a 
{\it finite split embedding problem}.
A {\it weak solution} to an embedding problem $(\beta, \alpha)$ is a group homomorphism 
$\lambda: \Pi \to \Gamma$ such that $\alpha \circ \lambda = \beta$.
A weak solution $\lambda$ is a {\it proper solution} if it is a surjection.
A group $\Pi$ is {\it projective} if every finite embedding problem for $\Pi$ has a weak solution.
Notice that any finite split embedding problem automatically has a weak solution given by $s \circ \beta$.

Solutions to embedding problems are tightly connected to the property of freeness.
For example, given a profinite group $\Pi$ and an infinite cardinal $m$,  
by \cite[Chapter 8]{rz}, $\Pi$ is free of rank $m$ if and only if 
every finite embedding problem has exactly $m$ distinct proper solutions.  
Moreover, by \cite[Thm.\ 2.1]{HS05} this is equivalent to $\Pi$ being projective and satisfying the property that every non-trival finite split embedding problem for $\Pi$ has exactly $m$ distinct solutions.  

\begin{proof} (Outline following \cite{har95, pop}.) 
Let $(\beta: G_{k(X)} \twoheadrightarrow G, \alpha: \Gamma \twoheadrightarrow G)$  
be a finite embedding problem for $G_{k(X)}$ and let $N$ be the kernel of $\alpha$. 
It suffices to show that $(\beta, \alpha)$ has ${\rm card}(k)$ distinct proper solutions.  
The surjection $\beta$ corresponds to a $G$-Galois cover $\phi:Y \to X$ whose branch locus is contained in a non-empty set $B \subset X$. 
Thus there is an induced embedding problem 
$(\beta: \pi_1(X-B) \twoheadrightarrow G, \alpha: \Gamma \twoheadrightarrow G)$. 
As in Theorem \ref{Tserre}, representation theory implies that obstructions 
for weak solutions lie in $H^2(\pi_1(X-B, N)$.  
By \cite[Prop.\ 1]{serre} the fundamental group of an affine 
$k$-curve has cohomological dimension at most $1$, 
and hence it is a projective group \cite[I.5.9, Prop.\ 45]{serregc}.
As a result, there exists a weak solution $\lambda: \pi_1(X-B) \to \Gamma$.
This defines a (possibly disconnected) $\Gamma$-Galois cover $\psi': Z' \to X$ with branch locus in $B$.  
The cover $\psi'$ can be patched with a branched $N$-Galois cover $Z'' \to \PP_k$ 
in such a way as to produce a proper solution.  
Specifically, the patched cover produces a (connected) $\Gamma$-Galois cover 
$\psi:Z \to X$ dominating $\phi$ which is unramified away from a finite set $B'$ containing $B$.
\end{proof}

In the proof above, the additional branching of the $\Gamma$-Galois cover $\psi$ at $B'-B$ 
is not a problem as $\psi$ still corresponds to a surjection $\lambda:G_{k(X)} \twoheadrightarrow \Gamma$.  
It is useful to remark here that if the kernel $N$ is a quasi-$p$ group then no additional ramification is required (i.e., $B' = B$) \cite[Thm.\ 4.6]{har03} because the wild ramification can be enlarged at one point.  
In general, the number of additional branch points 
depends on the number of generators of the maximal prime-to-$p$ quotient of $N$.
Moreover, the location of the additional branching cannot be prescribed.
As a result, the freeness result Theorem \ref{Tfree} for $G_{k(X)}$ does not translate into a freeness result for $\pi_1(X-B)$ because in the latter case the covers cannot have additional ramification outside of $B$.
It is still true that $\pi_1(X-B)$ has cohomological dimension at most 1, and thus it is projective.  
However, as we saw at the end of Section \ref{pro-p}, it is not free. 

\section{Open questions and results}\label{Sopen}

Let $k$ be an algebraically closed field of characteristic $p >0$.
At this time, the full structure of the fundamental group is not known for any affine $k$-curve or 
for any projective curve of genus $g \geq 2$.
The fundamental group depends on towers of covers of $k$-curves and on the geometry of the 
$k$-curves in these towers.  
The goal of understanding fundamental groups 
provides a strong motivation to answer new questions about these towers.  
Let $X$ be a smooth connected projective $k$-curve of genus $g$.  Let $B \subset X$ be a finite set of points. 

\subsection{Subgroups of fundamental groups of curves}

It is interesting to measure the extent to which the fundamental group of a $k$-curve is not free.  
In this section, we study this topic in terms of subgroups of the fundamental group 
and in the next section we will address this topic in terms of quotients of the fundamental group and 
embedding problems.

\begin{question} \label{Qfree}
If $X-B$ is an affine $k$-curve,
which closed normal subgroups of $\pi_1(X-B)$ are free?
\end{question}

For example, the commutator subgroup of $\pi_1(X-B)$ is free \cite[Thm.\ 6.12]{ku} for every affine $k$-curve $X-B$.
This is a natural subgroup to study for this question since the quotient of $\pi_1(X-B)$ by the commutator subgroup is the maximal abelian quotient of $\pi_1(X-B)$.
Additional examples of a similar type can be seen in \cite[Thm.\ 1.1]{psz}.  

Here is an example of an affine $k$-curve $X-B$ and a closed normal subgroup of $\pi_1(X-B)$ that is not free.

\begin{example}  Consider the affine line $\Aa_k$. 
Let $N$ be the intersection of all open normal subgroups of $\pi_1(\Aa_k)$ of index $p$ such that,
for the corresponding cover $Y \to \PP_k$ branched only at $\infty$, the curve $Y$ has genus zero.  
These subgroups correspond to Artin-Schreier covers $y^p-y = cx$ with $c \in k$.  
A computation shows that two such covers are linearly disjoint as long as $c_1-c_2$ 
is not a $(p-1)^{st}$ root of unity.
Thus $N$ has infinite index and it is a closed normal subgroup of $\pi_1(\Aa_k)$.

Let $k(x)_{\infty}$ be the maximal Galois extension of the function field $k(x)$ 
which is unramified outside $\{\infty\}$.
Then ${\rm Gal}(k(x)_{\infty}/k(x)) = \pi_1(\Aa_k)$.
Let $F_N$ be the fixed field of $N$ in $k(x)_{\infty}$.

Assume that $N$ is free of (possibly infinite) rank $r$.  
Then, for any finite group $G$ with at most $r$ generators, 
there would exist a surjection $\beta: N \twoheadrightarrow G$.
Such a surjection would correspond to a $G$-Galois field extension $\CL_{\beta}/F_N$.
Since $G$ is finite, this extension and its Galois action 
would be defined by a finite set $S$ of polynomials with coefficients in $F_N$.  
Thus $S$ would be defined over some field $E_{\beta}$ where $k(x) \subset E_{\beta} \subset F_N$ 
and where $E_{\beta}/k(x)$ has finite degree.
The Galois group of $k(x)_\infty$ over $E_{\beta}$ is an open subgroup $N_{\beta}$ of $\pi_1(\Aa_k)$.
Since $E_{\beta} \subset F_N$, one sees that $N \subset N_{\beta}$. 
Moreover, there exists a $G$-Galois extension $L_{\beta}/E_{\beta}$ such that $L_{\beta} \subset k(x)_{\infty}$ and $L_{\beta} \cdot F_N = \CL_{\beta}$. 
This process is called ``descending'' the extension.

Another computation shows that the fiber product of two linearly disjoint covers 
$y^p-y =c_1x$ and $y^p-y = c_2x$ yields a cover $Y \to \PP_k$ totally ramified over $\infty$ where $Y$ again 
has genus $0$.
Thus, for any open normal subgroup of $\pi_1(\Aa_k)$ containing $N$, 
the corresponding cover of $\PP_k$ has genus zero and is totally ramified over $\infty$.  
In particular, consider the (not necessarily Galois) cover $U_\beta \to \PP_k$ 
corresponding to the subgroup $N_\beta$.  
Then $U_{\beta}$ has genus $0$ and the fibre of $U_\beta$ over $\infty$ consists of one point $P_\infty$.
Since $E_{\beta}$ is the function field of $U_\beta$, 
the existence of a $G$-Galois extension $L_{\beta}/E_{\beta}$ with $L_{\beta} \subset k(x)_{\infty}$ 
implies that there exists a $G$-Galois cover $V \to U_{\beta}$ branched only at $P_\infty$. 
Choosing $G$ to be prime-to-$p$ and generated by $r \geq 1$ elements, 
this leads to a contradiction with the fact that $\pi^{p'}_1(\Aa_k)$ is trivial.
\end{example}

\subsection{Quotients of fundamental groups of curves}

Another approach to understanding the fundamental group is to study how its finite quotients fit together 
by solving embedding problems. 
This topic is especially important for the fundamental group of a $k$-curve $X$ which is projective. 
The reason is that, when $X$ is projective, then $\pi_1(X)=\pi_1^t(X)$ 
and so Theorem \ref{GR} implies that $\pi_1(X)$ is finitely generated as a profinite group.  
As such, by \cite[Prop.\ 15.4]{FJ}, it is determined by its finite quotients
(i.e., it is determined by the answer to question (1) from Section \ref{fundgrp}). 
Unfortunately, Abhyankar's Conjecture (Theorem \ref{AC}) does not apply to projective curves and
for $g \geq 2$, the finite quotients of $\pi_1(X)$ are unknown.
However, by Theorems \ref{GR} and \ref{ds} 
the maximal prime-to-$p$ and pro-$p$ quotients of the fundamental group of every projective $k$-curve are known.
Thus the question becomes, how do these prime-to-$p$ and pro-$p$ quotients fit together? 

A first step is to determine which finite groups $G$ having a normal 
$p$-Sylow subgroup $P$ occur as a quotient of $\pi_1(X)$.  
Such a quotient corresponds to an unramified $G$-Galois cover $Z \to X$ which factors as $Z \to Y \to X$ where $\Gal(Z/Y)=P$ and $\Gal(Y/X)$ is the prime-to-$p$ group $H = G/P$.  
In \cite[Thm.\ 7.5]{ps}, a necessary and sufficient condition is given for such $G$-Galois covers of $X$ to occur. The result essentially says that the $H$-module structure of $P$ must be compatible with the 
$H$-module structure of $J_Y[p]$ for some $H$-Galois cover $Y \to X$.  
The compatibility is measured in terms of a generalization of the $p$-rank called the 
Hasse-Witt invariants \cite[Section 2]{ruck} of $Y$.
Since $H$ is prime-to-$p$ and the prime-to-$p$ quotients of $\pi_1(X)$ are known, 
this result gives insight into the structure of $\pi_1(X)$.
The result was extended by Borne \cite[Thm.\ 1.1]{borne} to the case where $|H|$ 
is not necessarily prime-to-$p$.  
The proof in that case uses modular representation theory.

Nevertheless, the structure of $\pi_1(X)$ and its finite quotients are still unknown when $g \geq 2$.  
A complete analysis of this problem seems beyond reach for now.  
The results in \cite{borne} and \cite{ps} give conditions to solve the embedding problems when the kernel is a $p$-group. Thus, there is a natural question to ask next. 

\begin{question} \label{Qproj}
Given a projective curve $X$ and an embedding problem 
$(\beta: \pi_1(X) \twoheadrightarrow G, \alpha: \Gamma \twoheadrightarrow G)$ with $|{\rm ker}(\alpha)|$ prime-to-$p$, what conditions on $\Gamma$ and $X$ will ensure the existence of a proper solution?
\end{question}

\subsection{Ramification of covers of curves}

Given $X$, $B$, and $G$, only in special cases is it known what ramification data can occur 
for $G$-Galois covers $\phi:Y \to X$ with branch locus in $B$.
Answering this question is necessary to determine which values will occur for the genus of $Y$.
This is important for the goal of understanding the fundamental group $\pi_1(X)$, 
because the finite quotients of $\pi_1(Y)$ will depend on invariants like the genus or the $p$-rank of $Y$. 

This topic is most interesting for the case of wildly ramified covers of affine $k$-curves
because, in this case, there is the extra structure of the filtration of higher ramification groups to consider.
One result is that a cover can always be deformed using formal patching 
to lengthen the filtration of higher ramification groups at a wildly ramified point.  
Since the degree of the different depends on 
the ramification filtration, this leads to the following result.

\begin{theorem} \cite[Cor.\ 3.4]{Pr:genus}
Suppose $X-B$ is an affine $k$-curve and $G$ is a finite quotient of $\pi_1(X-B)$ 
such that $p$ divides $|G|$.  Let $N \in \NN$.
Then there exists a $G$-Galois cover $\phi:Y \to X$ with branch locus in $B$ such that the 
genus of $Y$ is greater than $N$.
\end{theorem}

An open problem is to determine the smallest genus that can occur for a 
$G$-Galois cover of $X$ with branch locus in $B$.
Because of results like Proposition \ref{Psupsing}, the smallest genus will often depend on the subset $B$, 
not just on its cardinality.

\begin{question} \label{Qgenus}
Given an affine $k$-curve $X-B$ and a finite quotient $G$ of $\pi_1(X-B)$, what is the smallest positive 
integer $g=g(X,B,G)$ which occurs as the genus of $Y$ for a $G$-Galois cover $\phi:Y \to X$ with branch locus in $B$?
\end{question}

A crucial case is to understand Galois covers of the affine line.
By Abhyankar's Conjecture, there exists a $G$-Galois cover of the affine line
if and only if $G$ is quasi-$p$, which means that $G$ is generated by $p$-groups.
For the affine line, if $G$ is an abelian $p$-group, then 
the answer to Question \ref{Qgenus} can be determined by class field theory.   
There are many other quasi-$p$ groups, including all simple groups with order divisible by $p$.
When $G$ is the projective special linear group ${\rm PSL}_2(\FF_p)$, then the answer to 
Question \ref{Qgenus} is $(p-1)^2/4$, \cite{BW}.
Under certain group theoretic conditions, an upper bound for the minimal genus can be found in 
\cite[Thm.\ 3.5]{Pr:jump1}.

One example of a quasi-$p$ group is a non-abelian semi-direct product $G$ of the form
$(\ZZ/\ell)^a \rtimes \ZZ/p$ where $\ell$ and $p$ are distinct primes and $a$ is the order of $\ell$ modulo $p$.
In a group project supervised by the authors at the WIN conference in Banff, November 2008, 
the group calculated the minimal genus that can occur for a Galois cover of the affine line with this group $G$.  
Specifically, in \cite[Thm.\ 4.1]{Banff}, the group proved that there is a 
$\Ga$-Galois cover $Z \to \PP_k$ branched only at $\infty$ with genus $g_Z=1+\ell^a(p-3)/2$
if $p$ is odd.
In addition, the group proved that this is the minimal genus and that 
there are only finitely many curves of this minimal genus 
which are Galois covers of the affine line with this Galois group.  
For the proof, the group determined the action of an automorphism of order $p$ on $J_Y[\ell]$ where 
$Y$ is the Artin-Schreier curve $y^p-y=x^d$.
This gave insight into the unramified elementary abelian $\ell$-group covers of $Y$ that are Galois over $\PP_k$.
We now extend this result to a more general class of quasi-$p$ groups.

For a finite group $G$, let $\Phi(G)$ denote the Frattini subgroup of $G$ 
(the intersection of all proper maximal subgroups of $G$).  This is the set of ``non-generators" of $G$.
If $\ell$ is a prime and $L$ is an $\ell$-group then $\Phi(L) = L^{\ell}[L,L]$ and $\mathcal{L}=L/\Phi (L)$ is an elementary abelian $\ell$-group.  
We will need the following lemma.

\begin{lemma} \label{Lsemidirect}
Let $\ell$ and $p$ be distinct primes and let $a$ be the order of $\ell$ modulo $p$.
There is a unique non-abelian semi-direct product of the form $(\ZZ/\ell)^a \rtimes \ZZ/p$ up to isomorphism.
\end{lemma}

\begin{proof}
Let $G$ be a non-abelian semi-direct product of the form $(\ZZ/\ell)^a \rtimes \ZZ/p$.
Then $G$ is determined by a non-trivial homomorphism $\gamma:\ZZ/p \to \Aut((\ZZ/\ell)^a)$.
The isomorphism type of $G$ depends only on ${\rm Im}(\gamma)$ because of the flexibility of choice of a generator for $\ZZ/p$.
Furthermore, it depends only on the conjugacy class of ${\rm Im}(\gamma)$ because of the choice of basis for $(\ZZ/\ell)^a$.
Thus, to show that $G$ is unique up to isomorphism, it suffices to show that all subgroups of order $p$
in $\Aut((\ZZ/\ell)^a) \simeq {\rm GL}_a(\ZZ/\ell)$ are conjugate. 
Let $H \subset {\rm GL}_a(\ZZ/\ell)$ be a subgroup of order $p$ and let $h \in H$ be a generator.
Up to conjugacy, $h$ can be chosen in rational canonical form.
Since $a$ is the order of $\ell$ modulo $p$, the vector space $(\ZZ/\ell)^a$ is indecomposable under the
semi-direct product action.
The matrix $h$ consists of one block since the action is indecomposable.
Thus $h$ is determined by its characteristic polynomial $f_h(x)$. Then $f_h(x)$ is an irreducible (degree $a$) factor of the cyclotomic polynomial $\Phi_p(x)$. After possibly changing the generator $h \in H$, then $f_h(x)$ is the minimal polynomial for a fixed $p$th root
of unity $\zeta_p$.  Thus the conjugacy class of $H$ is uniquely determined.
\end{proof}

Here is the answer to Question \ref{Qgenus} for groups of the form $L \rtimes \ZZ/p$ where 
$L$ is an $\ell$-group whose maximal elementary abelian quotient is $(\ZZ/\ell)^a$.

\begin{proposition} \label{Pgenus}
Let $\ell$ and $p$ be distinct primes with $p$ odd.
Suppose $L$ is an $\ell$-group such that the quotient $L/\Phi(L)$ is elementary abelian of rank $a = ord_p(\ell)$.
Suppose $\Gamma$ is a quasi-$p$ group which is a semi-direct product of the form $L \rtimes \ZZ/p$.
Then there exists a $\Gamma$-Galois cover $W \to \PP_k$ branched only at $\infty$ such that the genus of $W$ is $g_W=1+|L|(p-3)/2$.
This is the minimal genus that occurs for a $\Gamma$-Galois cover of $\PP_k$ branched only at $\infty$.
\end{proposition}

Before proving Proposition \ref{Pgenus}, we need some information about Frattini covers.
A surjective group homomorphism $\phi: G \twoheadrightarrow H$  
is a {\it Frattini cover} if $\ker(\phi) \subset \Phi(G)$.
For each finite (even profinite) group $H$, there exists a cover $\tilde{\phi}:{\mathcal H} \to H$, unique up to isomorphism, such that $\tilde{\phi}$ is the largest Frattini cover of $H$.  The group $\mathcal H$ is the {\it universal Frattini cover} of $H$  (see \cite[Chapter 20, sections 6 and 7]{FJ} or \cite[22.11 and 22.12]{FJ2} for definitions and details). 
A group $N$ is a normal subgroup of $\mathcal H$ if and only if it is a Frattini cover of $H$.

The universal Frattini cover of $H$ is in fact the smallest cover of $H$ that is projective.
In other words, every embedding problem $(\tilde{\phi}: {\mathcal H} \twoheadrightarrow H, \alpha: G \twoheadrightarrow H)$ has a {\sl weak} solution $\lambda$.  
When $\alpha$ is a Frattini cover, then $\lambda$ is automatically a proper solution (i.e., surjective).  

\begin{proof}
The group $\Gamma$ has a quotient $H$ which is a semi-direct product $(\ZZ/\ell)^a \rtimes \ZZ/p$.
Since $\Gamma$ is quasi-$p$, the group $H$ is non-abelian.  
By Lemma \ref{Lsemidirect}, the structure of $H$ is uniquely determined up to isomorphism.

Let $\CL$ be the universal Frattini cover of $(\ZZ/\ell\ZZ)^a$. 
This is a free pro-$\ell$ group of rank $a$.
Because $\CL/\Phi(\CL)= (\ZZ/\ell\ZZ)^a$ and $\Phi(\CL)$ is the set of non-generators of $\CL$, 
the infinite group $\CL$ can be generated by $a$ elements.

The semi-direct product $H$ is determined by an action of $\ZZ/p$ on $(\ZZ/\ell\ZZ)^a$.
This induces an action of $\ZZ/p$ on $\CL$ \cite[Prop. 22.12.2]{FJ}.
Let $\CL \rtimes \ZZ/p$ be the resulting semi-direct product.  
Then $\CL \rtimes \ZZ/p$ is the universal Frattini $\ell$-cover of $\Ga$
and $\Gamma$ is a quotient of $\CL \rtimes \ZZ/p$.
That is, there exists a normal subgroup $N$ of $\CL$ that is 
$\ZZ/p$-invariant with $(\CL/N) \rtimes \ZZ/p = \Gamma$.

By \cite[Thm.\ 4.1]{Banff}, there is a $\Ga$-Galois cover $Z \to \PP_k$ branched only at $\infty$.
Furthermore, it factors through the Artin-Schreier cover $\phi:Y_2 \to \PP_k$ with equation $y^p-y=x^2$.
Also the $(\ZZ/\ell\ZZ)^a$-Galois cover $Z \to Y_2$ is unramified.
This yields a surjection $\psi_1: \pi_1(Y_2) \twoheadrightarrow (\ZZ/\ell\ZZ)^a$.
Since $\ell$ is prime-to-$p$ and $a = \mathrm{ord}_p(\ell) \leq p-1 =2g(Y_2)$, 
by \cite[Cor.\ 2.12]{Gr} there exists a surjection $\psi_2: \pi_{1}(Y_2)  \twoheadrightarrow \CL$ that dominates $\psi_1$.
This induces an infinite unramified $\CL$-Galois extension $F$ of the function field $k(Y_2)$ of $Y_2$.

As $k(Y_2)$ is a $\ZZ/p$-Galois extension of $k(x) = k(\PP_k)$ branched only at $\infty$, 
the extension $F/k(x)$ is algebraic and branched only at $\infty$.
Let $F'$ be the Galois closure of $F/k(x)$.  
Then $F'/k(Y_2)$ is a Galois extension with pro-$l$ Galois group that surjects onto $\CL$ and thus also onto $(\ZZ/\ell\ZZ)^a$.
But $\CL$ is universal for all pro-$l$ groups surjecting onto $(\ZZ/\ell\ZZ)^a$ \cite[Remark 22.11.19]{FJ2}   so $F'=F$ and the extension $F/k(x)$ is Galois. 
By Schur-Zassenhaus the Galois group is $\CL \rtimes \ZZ/p$.
Thus there is a surjection $\psi_2':\pi_1(\Aa_k) \to \CL \rtimes \ZZ/p$.

Taking the composition of $\psi_2'$ with the natural surjection $\CL \rtimes \ZZ/p \to \Gamma$, 
this yields a surjection $\lambda: \pi_1(\Aa_k) \to L \rtimes \ZZ/p$. 
This induces an unramified $\Gamma$-Galois cover $W \to \PP_k$ branched only at $\infty$ and 
dominating $\phi$.  Moreover, the cover $W \to Y_2$ is unramified.

By the Riemann-Hurwitz formula, the genus of $W$ is $1+|L|(p-3)/2$.  
The statement that this is the minimal genus follows just as in \cite[Thm. 4.1]{Banff}, 
since the minimal genus will be realized when the $L$-Galois subcover is unramified and the genus of the 
$\ZZ/p$-Galois quotient is the smallest positive number possible.
\end{proof}

\subsection{An open question on arithmetic invariants of Galois covers}

As discussed in Section \ref{JacCov}, there is a connection between unramified $\ZZ/p$-Galois 
covers of a projective curve and the $p$-torsion of its Jacobian.  
As a result (see Theorem \ref{Ttama}),
the fundamental group $\pi_1(X-B)$ will depend on the $p$-rank $s_Y$ when $\phi:Y \to X$ is a 
Galois cover with branch locus in $B$.  
For this reason, there is good motivation to understand the values that occur for the $p$-rank 
associated with covers.
Even for the case when $G$ is cyclic and $X=\PP_k$, 
there are many papers on this subject, e.g., \cite{Bo01}, \cite{Y}.  

There are arithmetic invariants of the Jacobian of a $k$-curve other than its $p$-rank, including 
the Newton polygon and the $p$-torsion group scheme (see \cite{demazure} and \cite{O:strat} respectively).
As an example, recall that an elliptic $k$-curve $E$ can be either ordinary or supersingular.
The two cases can be distinguished by the number of points in $E[p](k)$, which is either $p$ or $1$.
If $E$ is ordinary, then its Newton polygon has slopes $0$ and $1$.  The $p$-torsion group scheme
of an ordinary elliptic curve is $E[p] \simeq \ZZ/p \oplus \mu_p$
where $\mu_p$ is the kernel of Frobenius on $\GG_m$.
If $E$ is supersingular, then its Newton polygon has slopes $1/2$.
The $p$-torsion group scheme of a supersingular elliptic curve fits into a (non-split) short
exact sequence $1 \to \alpha_p \to E[p] \to \alpha_p \to 1$ where  
$\alpha_p$ is the kernel of Frobenius on $\GG_a$. 
Let $E_{ss}[p]$ denote the (unique) isomorphism class of the $p$-torsion group scheme 
of a supersingular elliptic curve.

While the connection between these other invariants and the fundamental group is not clear, it still raises the 
following question.

\begin{question} \label{Qprank}
Given a finite group $G$ which is a quotient of $\pi_1(X-B)$, what are the possibilities for 
the $p$-rank, Newton polygon, and $p$-torsion group scheme of $J_Y$ for $G$-Galois covers $\phi:Y \to X$
with branch locus in $B$?
\end{question}

Here is a new result about this question, building upon the group result in \cite[Thm.\ 4.1]{Banff}.
We find a Galois cover $Z \to \PP_k$ branched only at $\infty$ 
with Galois group $\Ga$ such that $Z$ has small genus and large $p$-rank.

\begin{proposition} \label{Pprank}
Let $\ell$ and $p$ be distinct primes with $p$ odd and $\ell \geq -1+ (p-1)^2/2$.
Let $a$ be the order of $\ell$ modulo $p$.
Suppose $G$ is the non-abelian semi-direct product $(\ZZ/\ell)^a \rtimes \ZZ/p$. 
Then there exists a Galois cover $Z \to \PP$ branched only at $\infty$, with Galois group $G$, genus
$g_Z=1+\ell^a(p-3)/2$ and $p$-rank $s_Z=(\ell^a-1)(p-3)/2$.
Furthermore, $J_Z[p]$ decomposes completely into $s_Z$ copies of $\ZZ/p \oplus \mu_p$ and $(p-1)/2$ copies 
of $E_{ss}[p]$, the $p$-torsion group scheme of a supersingular elliptic curve.
In particular, the Newton polygon of $J_Z$ only has slopes $0$, $1/2$, and $1$.
\end{proposition}

\begin{proof}
Consider the cover $\phi:Y \to \PP_k$ with affine equation $y^p-y=x^2$.  
Then $Y$ has genus $g_Y=(p-1)/2$ and $p$-rank $0$.
By \cite[Cor.\ 3.3]{Pr}, 
$J_Y$ is superspecial, i.e., $J_Y[p]$ decomposes into $g_Y$ copies of $E_{ss}[p]$.
In particular, $J_Y$ is supersingular, i.e., the slopes of its Newton polygon all equal 1/2. 

If $Z_1 \to Y$ is an unramified $\ZZ/\ell$-Galois cover, then $Z_1$ has genus $g_{Z_1}=1+\ell(p-3)/2$
by the Riemann-Hurwitz formula.  
Suppose $\ell \not = p$ is prime such that $\ell +1 \geq (p-1)^2/2$.
By \cite[4.3.1]{ray1982}, there exists an unramified $\ZZ/\ell$-Galois cover 
$Z_1 \to Y$ such that the new part of $J_{Z_1}$ is ordinary.
Thus $Z_1$ has $p$-rank $s_{Z_1}=(\ell-1)(p-3)/2$
and $J_{Z_1}[p]$ contains a factor isomorphic to $(\ZZ/p \oplus \mu_p)^{s_{Z_1}}$.

Also $J_Y$ is isogenous to a factor of $J_{Z_1}$. 
Since the cover $Z_1 \to Y$ has degree $\ell$, the degree of the isogeny is prime-to-$p$.
As a result, $J_{Z_1}[p]$ contains a factor isomorphic to $J_Y[p]$.
Thus $J_{Z_1}[p]$ decomposes into $s_{Z_1}$ copies of $(\ZZ/p \oplus \mu_p)$ and $(p-1)/2$ copies of $E_{ss}[p]$. 
In particular, the Newton polygon of $J_{Z_1}$ has slopes $0$, $1/2$, and $1$.

Consider the action of an automorphism $\sigma$ of $Y$ of order $p$ on the set of unramified 
cyclic $\ZZ/\ell$-Galois covers of $Y$.  If $Z_2$ is in the orbit of $Z_1$ under the action of $\sigma$, then 
$Z_2$ and $Z_1$ are isomorphic, and so every invariant of the curves is the same.
Consider the Galois closure $\psi: Z \to Y \to \PP$ of $Z_1 \to Y \to \PP$.
The Galois group of $\psi$ is isomorphic to $G$ since it is 
a semi-direct product of the form $(\ZZ/\ell)^a \rtimes \ZZ/p$, by Lemma \ref{Lsemidirect}.  
The genus is $g_Z=1+\ell^a(p-3)/2$ by the Riemann-Hurwitz formula.

Relative to the cover $Z \to Y$, the new part of $J_Z[p]$ is ordinary 
and the old part of $J_Z[p]$ is isomorphic to $Y[p]$.  
Thus the curve $Z$ has $p$-rank $s_Z=(\ell^a-1)(p-3)/2$ and
$J_Z[p]$ decomposes completely into $s_Z$ copies of $\ZZ/p \oplus \mu_p$ and $(p-1)/2$ copies of $E_{ss}[p]$.  
In particular, the Newton polygon of $J_Z$ has slopes $0$, $1/2$, and $1$.
\end{proof}

\bibliographystyle{plain}	
\bibliography{winsurvey}	

\end{document}